\documentclass[a4paper]{article}
\usepackage[utf8x]{inputenc}

\usepackage{enumerate}
\usepackage{comment}

\usepackage{amsmath}
\usepackage{amsfonts}
\usepackage{graphicx}
\usepackage{url}
\usepackage{enumitem}

\usepackage{amsthm}
\newtheoremstyle{indenteddefinition}{\topsep}{\topsep}{\setlength{\leftskip}{\parindent}}{-1em}{\bfseries}{.\mbox{} \\}{5pt plus 1pt minus 1pt}{}
\theoremstyle{indenteddefinition}
\newtheorem{thm}{Theorem}
\newtheorem{lem}{Lemma}
\newtheorem{defn}{Definition}

\newtheoremstyle{indentedproof}{\topsep}{\topsep}{\setlength{\leftskip}{\parindent}}{-1em}{}{\mbox{} \\}{5pt plus 1pt minus 1pt}{\emph{#3}:}
\theoremstyle{indentedproof}
 \newtheorem*{pf}{}

\makeatletter
\renewenvironment{proof}[1][\proofname]
  {\par\pushQED{\qed}\begin{pf}[#1]\ignorespaces}
  {\popQED\end{pf}}
\makeatother

\newcommand{\F}{\ensuremath{\mathbb{F}}}
\newcommand{\N}{\ensuremath{\mathbb{N}}}

\newcommand{\Q}{\ensuremath{\mathbb{Q}}}

\newcommand{\R}{\ensuremath{\mathbb{R}}}

\newcommand{\co}{\alpha} 
\newcommand{\propp}{P}
\newcommand{\proppF}{P_\F }
\newcommand{\ADF}{F}
\newcommand{\ADR}{\bar{F}}
\newcommand{\abR}{[a,b]}
\newcommand{\abF}{[a,b]_\F}
\newcommand{\oeR}{[0,1]}
\newcommand{\oeF}{[0,1]_\F}
\newcommand{\RA}{\noindent$(\Rightarrow)$: }
\newcommand{\LA}{\noindent$(\Leftarrow)$: }

\title{Would Real Analysis be complete without the Fundamental Theorem of Calculus?}
 \author{
 Michael Deveau (University of Waterloo, Canada)
 \mbox{} \\
 \mbox{}
 Holger Teismann (Acadia University, Canada)
 }
 \date{June 17, 2015}

\begin{document}

\maketitle

Echoing L.R.Ford's opening words\footnote{``Perhaps the author owes an apology to the reader for asking him to lend his attention to so
elementary a subject, for the fractions to be discussed in this paper are, for the most part, the
halves, quarters, and thirds of arithmetic."} of his delightful \textit{Monthly} article \cite{ford1938}, perhaps we, too, owe an apology to the reader for asking 
a seemingly flippant question in the title of this paper, whose answer must so obviously be `no'.  \linebreak After all, the adjective
``fundamental" says it all -- even if, as Bressoud points out,  that designation did not come into use until relatively recently
\cite{bressoud2011}. We admit that we chose the title for effect, accepting the possibility of leading the reader astray; a more descriptive 
title would have been:  ``would the real \textit{numbers} be complete without the Fundamental Theorem of Calculus?"
In some sense, however, the title is actually accurate in that this paper will show that a 
mathematical ``world" (which we interpret to mean 
``totally ordered field") 
without the Fundamental Theorem of Calculus would necessarily be lacking of many of the most cherished parts of Real 
Analysis.

Over the last decade or so it has been noticed that many statements / theorems from the standard canon of (single-variable) Real Analysis not only crucially depend on the completeness of 
the real numbers but are in fact \emph{equivalent} to completeness; see \cite{riemenschneider2001,propp2013,teismann2013} for various lists of such statements. 
While this fact may be reasonably well known, or at least be somewhat expected, for some statements such as the Intermediate, Mean, and  Extreme Value Theorems, it  may be more surprising for others, such as the Ratio Test \cite{propp2013}, the Principle of Real Induction \cite{clark2012} 
or the Weierstrass Approximation Theorem.  The most surprising feature of the list of statements equivalent to completeness, however, may very well be its sheer size, which, 
in its most recent version \cite{fiftyPlus2013}, comprises 70 items!

Curiously, the most coveted candidate for the list, the Fundamental Theorem of Calculus (FTC), has been the most ``difficult customer'' in this enterprise and 
so far resisted inclusion\footnote{Although \#30 in \cite{riemenschneider2001} and  \#16 in \cite{teismann2013} read identical to Theorem \ref{T2} below, we find their 
treatment with regard to completeness unsatisfactory.}. 

As evidence of the ``prickliness'' of the FTC, consider the simple  function $f:\Q \to \Q $, $f(x) = 1/(1+x^2)$. 
Although $f$ is a perfectly nice  function -- it is uniformly continuous, for example -- it is not Riemann-integrable 
over $[0,1]$, since the value of its integral\footnote{For the definition of the Riemann integral in subfields of the reals, see the last paragraph of Section \ref{Prelim}.} 
would have to be $\arctan (1) = \pi/4$, which is an irrational number. So what is the domain of the
``area function'' $F(x) = \int _0^x f(t)\,dt = \arctan(x)$ appearing in the FTC? Trying to identify it would amount to identifying the rational arguments $x$ for which $\arctan(x)$ is
a rational number, which would get one into deep water. What is more, it appears to be unknown whether $f$ even has an anti-derivative; 
i.e. whether there exists \emph{any} differentiable function $F:\Q \to \Q $ such that $F'(x) = f(x)$.

In this note we present two versions of the FTC, each of which turns out to be equivalent to the completeness of Archimedean fields (ordered subfields of $\R $). 
This is accomplished by separating the two aspects of the problem illustrated above: the first theorem deals with the mere existence of anti-derivatives, 
whereas the second one deals with the integrability of continuous functions and the domain of the area function. Let $\F$ be an arbitrary ordered subfield of $\R$. 

\begin{thm} \label{T1}
 $\F$ is complete if and only if every continuous function defined on a closed and bounded interval has a \emph{uniformly differentiable} anti-derivative.
\end{thm}
\begin{thm} \label{T2}
 $\F$ is complete if and only if every continuous function defined on a closed and bounded interval is Riemann-integrable 
 (consequently, its (signed) ``area function'' is defined on the whole interval).
\end{thm}
The bulk of this paper is devoted to the proof of the ``$\Leftarrow$'' direction of Theorem~\ref{T1}, which is done by contradiction.
More specifically, assuming that $\F$ is not complete, a continuous function (called ``P--function" for ``Propp function" below, as it is a slightly modified version of a function proposed by
J.~Propp \cite{propp2012}) is constructed, whose integral is a number in $\R \setminus \F$; this number would have to be the value of the anti-derivative at $x=1$, which 
is impossible. (In this respect, the P--function is similar to the function $f$ above.)
The assumption of \emph{uniform} differentiability is used to extend the various functions appearing in the proof from $\F$ to $\R$ and to utilize the standard FTC in $\R$. The P--function also provides the required counterexample for the ``$\Leftarrow$'' direction of Theorem~\ref{T2}. As an aside, we note that the strength of Theorem~\ref{T2}
does not change if ``continuous" is replaced with ``\emph{uniformly} continuous", as the P--function is actually uniformly continuous.
 
In view of Theorems \ref{T1} and \ref{T2}, the FTC has finally yielded and revealed its relationship to completeness --- well, actually, not quite: 
it is is still keeping one secret, namely  whether the assumption of \emph{uniform} differentiability in Theorem~\ref{T1} is really necessary, or if ``uniform" 
could be dropped  ``without penalty", i.e. without changing the strength of the theorem. The function $f:\Q \to \Q $ defined above is a case in point, as it still might -- just might -- have an anti-derivative\footnote{In view of Theorem~\ref{T1}, it cannot have a uniformly differentiable one, however. 
So if an anti-derivative does exist, it will have to be a ``fairly strange" one.}.    
 
This note is organized as follows. After briefly fixing some (minimal amount of) nomenclature and notation (Section~\ref{Prelim}), we present in Section~\ref{Propp} 
the construction of the P--function, which is the main technical device in the proofs of Theorems \ref{T1} and~\ref{T2}. 
In the next section we discuss the properties of uniformly differentiable functions as needed in the proofs (Theorem \ref{thm:UniformDiffRespectsExtension}); 
this section also contains another statement equivalent to completeness, which also involves uniform differentiability (Theorem~\ref{Complete}). In Section~\ref{Main} we finally present 
the proofs of the main results, which, thanks the preliminary work of the preceding sections, are actually pleasantly short. 
We end the paper with addressing the question that some readers may already have been wondering about: what about ``the other'' (part of the) FTC?  

\section{Preliminaries} \label{Prelim}
We assume that the reader is familiar with the basic definitions and properties of (totally) ordered fields. In this paper we will only consider \emph{Archimedean} (ordered) fields, 
which are known to be isomorphic to (ordered) subfields of the reals. We may therefore restrict ourselves to subfields of $\R $; accordingly, $\F$ will always denote such a subfield.
Examples of proper subfields of $\R $ include simple field extensions of $\Q$, such as $\Q(\sqrt{2})$, as well as the algebraic, computable, and constructible numbers. 
While all these fields are countable, it can be shown that $\R$ also contains proper subfields which are uncountable (see e.g. \cite{butcher1985}).  
Readers who still wonder -- or worry -- how large (or small) the class of subfields of $\R $ really is may take comfort in the knowledge that there are already uncountably many 
pairwise non-isomorphic subfields in the class of countable subfields alone \cite{schutt1989}.

As mentioned in the introduction, \emph{completeness} may be defined in many ways; the most widely-used definition is probably the one requiring the existence of suprema for bounded sets. 
Of course, any of the equivalent definitions will do, but  for the purposes of this paper, it is sufficient to interpret the completeness of $\F$ as $\F = \R$,
since $\R$ is obviously complete and any proper subfield of $\R $ is incomplete.

We adopt the convention that intervals $[a,b]$ \emph{without} subscripts refer to $\R$, whereas we typically add subscripts when referring to $\F$; i.e., for $a,b\in \F $, $a<b$,
we have $\abF = \abR \cap \F $.  Moreover, unless otherwise stated, $f$ always refers to a function $f : \abF \to \F $, and the function $\bar{f} :\abR \to \R $ denotes 
the (unique) extension of $f$, provided it exists.

Finally, we need to define the Riemann integral in subfields of the reals. Formally, the definition is identical to the usual one: functions are integrable if and only if the 
appropriate limits of Riemann sums converge \emph{in $\F$}, where, obviously, the Riemann sums are to be constructed \emph{within} $\F $; i.e. partition and sample points 
are numbers in $\F$.  The reader is cautioned, however, that this definition is different from the ones used in
\cite{riemenschneider2001} and \cite{teismann2013} (the latter being based on \cite{olmsted1973}, but still using different terminology).    

\section{Construction of the P--function}  \label{Propp}
\begin{defn}[\cite{propp2012}] \label{ProppFunction}
 Let $\F$ be incomplete and $\co$ in $\R \setminus \F$.
 We may assume w.l.o.g. that $\co$ is an irrational number in $\left(\frac{3}{8},\frac{5}{8}\right) \subset \R$.
 
 We will construct a continuous function $\proppF $ from $\oeR_\F =\oeR\cap \F $ to itself. 
 Let $0.d_1d_2d_3\ldots$ be a base $4$ representation of $\co - \frac{3}{8}$. That is,
 \[
   \co = \frac{3}{8} + \frac{d_1}{4} + \frac{d_2}{4^2} + \frac{d_3}{4^3} + \cdots = \frac{3}{8} + \sum_{j=1}^\infty \frac{d_j}{4^j}.
 \]
 with $d_j \in \left\{0,1,2,3\right\}$.
 
 Note however that since $\co \in \left(\frac{3}{8},\frac{5}{8}\right)$, we have $\co - \frac{3}{8} \in \left(0,\frac{1}{4}\right)$ and so $d_1$ must be $0$.
 
 Let $n \in \N$ be arbitrary and consider the square $S_n$ in $\oeR \times \oeR$ with corners $\left(2^{-n}, 2^{-n}\right)$ 
 and $\left(2^{-(n+1)}, 2^{-(n+1)}\right)$.
 We define $\propp _n : \left[2^{-(n+1)}, 2^{-n}\right] \to \left[2^{-(n+1)}, 2^{-n}\right]$ algebraically by requiring it to pass through the two points
 \begin{align*}
  Q_n &= \left(2^{-n}\left(1 - \frac{d_{n+2} + 1}{8}\right), 2^{-(n+1)}\right) \\
  R_n &= \left(2^{-n}\left(1 - \frac{d_{n+2}}{8}\right), 2^{-n}\right)
 \end{align*}
 and the two corners of $S_n$ defined above, and be linear on any open interval not containing any of these points. The resulting function is obviously continuous on its domain.
 
 We can also describe this geometrically. Clearly, the area of $S_n$ is $4^{-(n+1)}$. We divide $S_n$ as shown in Fig.~\ref{fig:PgramSn}. 
 \begin{figure}[ht]
  \centering\includegraphics[scale=0.4]{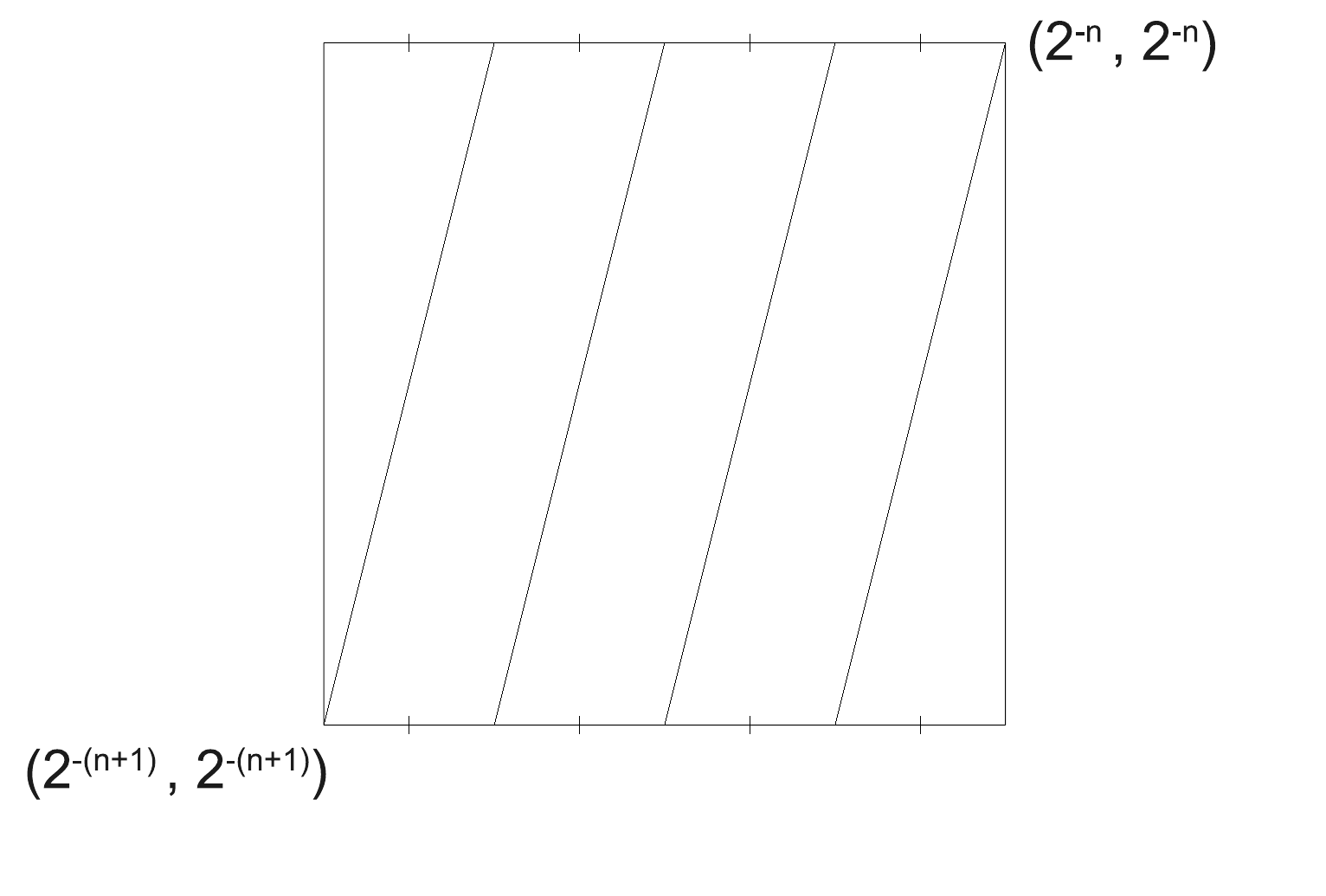}
  \caption[Construction of divisions of $S_n$.]{\label{fig:PgramSn}$S_n$ divided into equal regions, called \emph{parallelograms}.}
 \end{figure}
 Note that each parallelogram has area $4^{-n-2}$. By selecting a division of $S_n$  as in Fig. \ref{fig:DefnOfFn}, the division will have area $d_{n+2} 4^{-n-2}$.
 \begin{figure}[ht]
  \centering\includegraphics[scale=0.4]{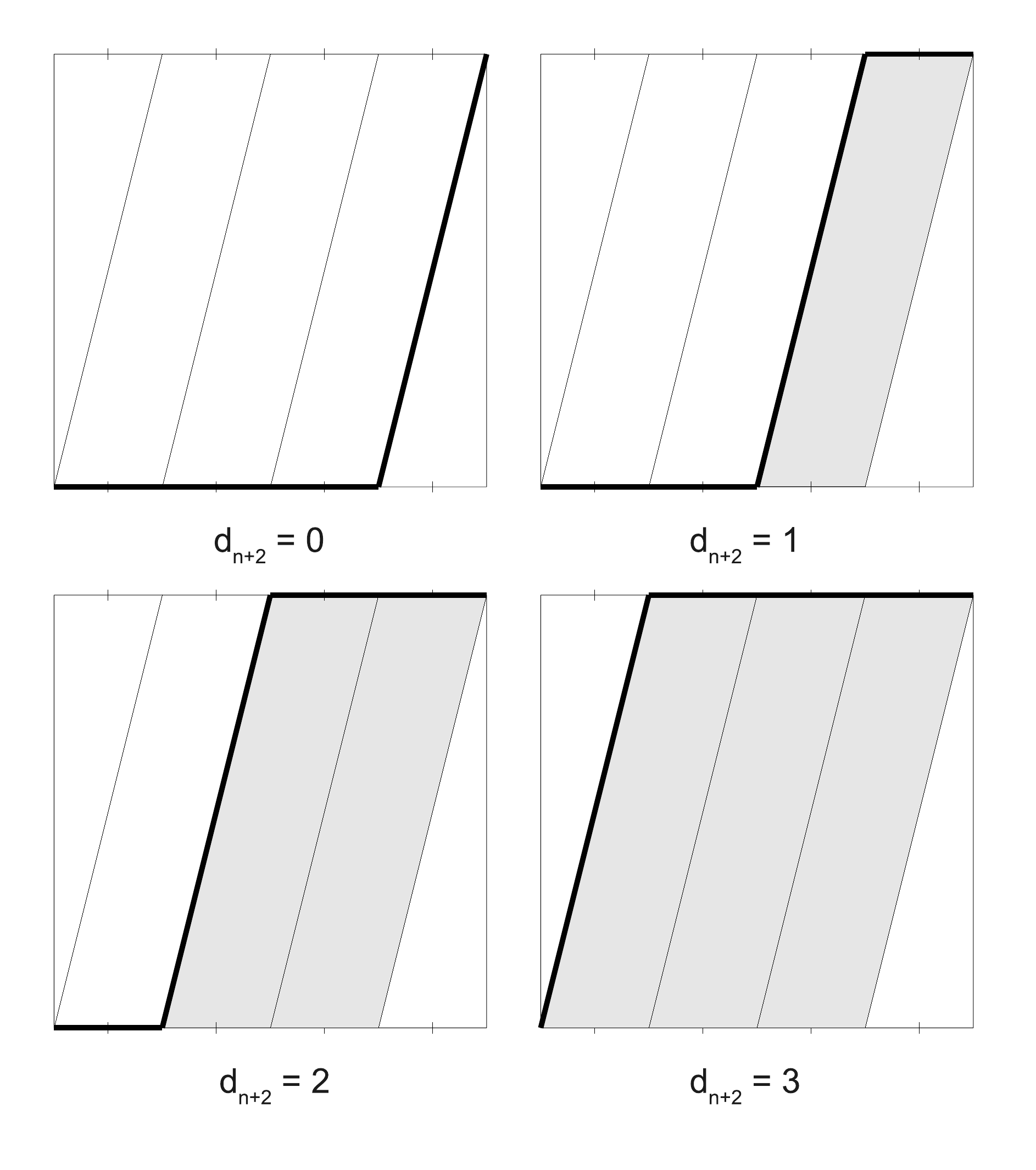}
  \caption[Construction of $\propp _n$ from $S_n$.]{\label{fig:DefnOfFn}Selection of regions in $S_n$ using the value of $d_{n+2}$.}
 \end{figure}
 We can treat the selected division as a function $\propp_n$ from $\left[2^{-(n+1)}, 2^{-n}\right]$ to itself shown by the darkened edges in Fig.~\ref{fig:DefnOfFn}.
 
 Let us now consider the definite integral of the function $\propp_n$ on its domain. Using the analytic description of $\propp_n$, a short computation yields 
 $\int_{2^{-(n+1)}}^{2^{-n}} \propp_n(x)\,dx = \left(d_{n+2} + 1/2\right)4^{-(n+2)} + 4^{-(n+1)}$.
 
 Now we define $\propp: \oeR \to \R$ by requiring it to be the continuous function passing through all $Q_n$, $R_n$, $(0,0)$ and $(1,1)$, and be linear on any open interval not
 containing any of these points. This is simply the piecewise concatenation of all $\propp _n$s. The function $\propp$ is continuous on $\oeR$ and so is uniformly continuous there.
 
 Finally, we define $\proppF : \oeR_\F \to \F$ to be the restriction of $\propp $ to $\oeR_\F$. Note that $\proppF$ is (well-defined and) uniformly continuous.
\end{defn}

As an example, suppose $c$ is an irrational number starting $0.4485\ldots$. Then $c = 3/8 + 1/4^2 + 0/4^3 + 2/4^4 + 3/4^5 + \cdots$, 
and so we get the function shown in Fig.~\ref{fig:ExampleFn}, accurate on $[1/16,1]$.
 
\begin{figure}[ht]
\centering\includegraphics[scale=0.5]{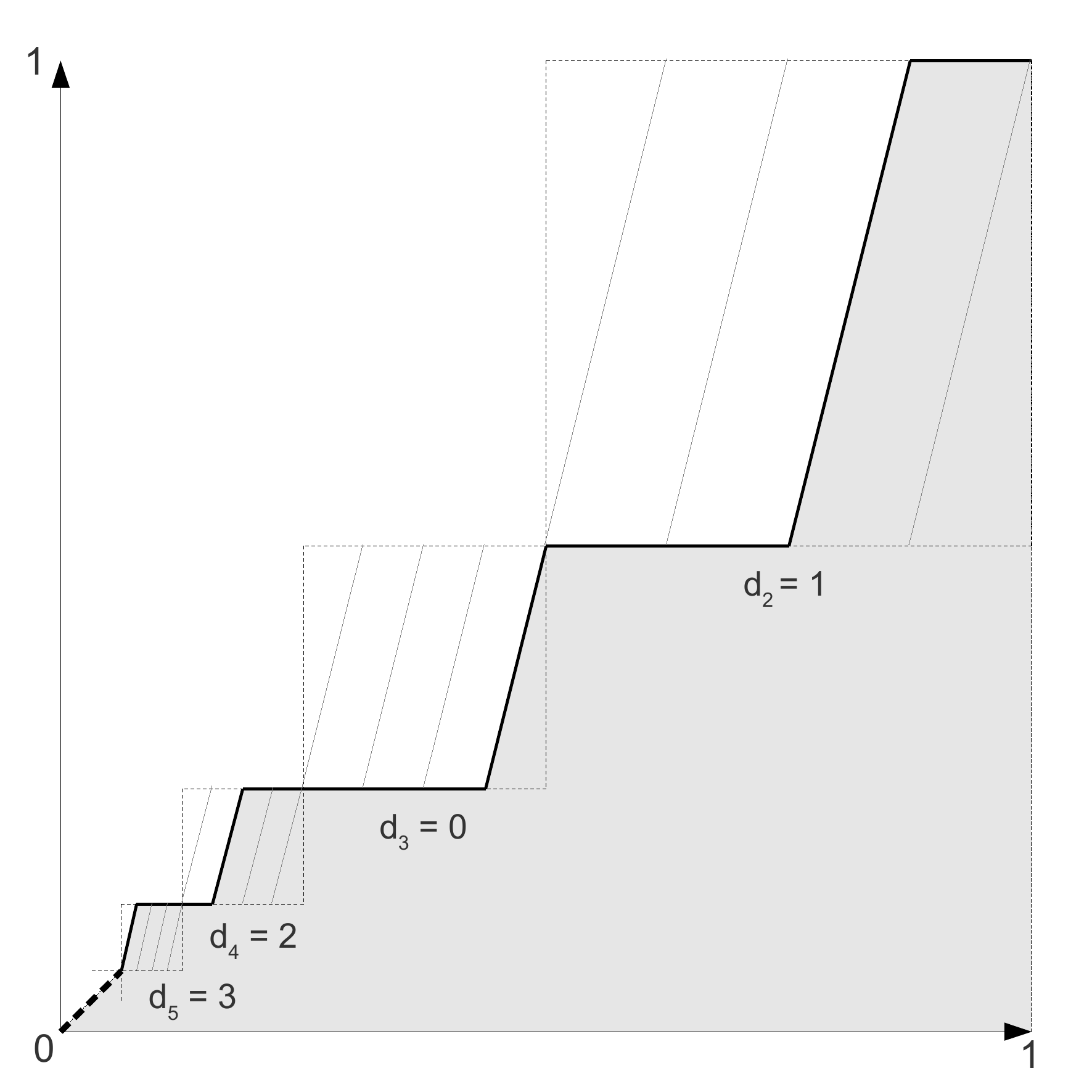}
\caption[Example $f$ using the Propp construction for $c= 0.4485\ldots$.]{\label{fig:ExampleFn}P--function for $c = 0.4485\ldots$.}
\end{figure}

\begin{lem} \label{lem:ProppIntegral}
 Let $\propp:\oeR \to \R$ be the unrestricted function defined in Def.~\ref{ProppFunction}. Then $\int_0^1 \propp(x)\,dx= \co$, where $\co \in \R \setminus \F $ is the value used in the 
 construction of $\propp$.
\end{lem}
\begin{proof}
 First, define, for $n \in \N$,
 \[
  t_n(x) \equiv \begin{cases}
          \propp |_{\left[2^{-(n+1)},1\right]}(x) &\mbox{if } x \in \left[2^{-(n+1)},1\right] \\
          0 &\mbox{otherwise}
         \end{cases}
 \]
 and note that $t_n$ is integrable with
 \[
  \int_0^1 \! t_n(x)\,dx = \sum_{j=0}^n \int_{2^{-(n+1)}}^{2^{-n}} \propp _j(x)\,dx.
 \]
 Now, $\{t_n\}$ converges uniformly to $f$ on $\oeR$, and so
 \begin{align*}
  \int_0^1 \! \propp(x)\,dx&= \int_0^1 \lim_{n \to \infty} t_n(x)\,dx= \lim_{n \to \infty} \int_0^1 \! t_n(x)\,dx \\
  &= \lim_{n \to \infty} \sum_{j=0}^n \int_{2^{-(n+1)}}^{2^{-n}} \propp_j(x)\,dx= \sum_{j=0}^\infty \int_{2^{-(n+1)}}^{2^{-n}} \propp_j(x)\,dx
 \end{align*}

 We computed the definite integral of each $\propp_j$ above and so
 \begin{align*}
  \int_0^1 \! \propp(x)\,dx&= \sum_{j=0}^\infty \left(d_{j+2} + \frac{1}{2}\right)4^{-(j+2)} + 4^{-(j+1)} \\
  &= \sum_{j=0}^\infty d_{j+2}\,4^{-(j+2)} + \left(\frac{1}{8} + 1\right)\sum_{j=0}^\infty\frac{1}{4}\left(\frac{1}{4}\right)^j \\
  &= \sum_{j=2}^\infty \frac{d_j}{4^j} + \frac{9}{8} \cdot \frac{1}{3} = \sum_{j=2}^\infty \frac{d_j}{4^j} + \frac{3}{8}
 \end{align*}
 which is $\co$ by definition. So $\int_0^1 \propp(x)\,dx = \co$ as desired.
\end{proof}

\section{Uniformly differentiable functions} \label{Diff}

In this section, we explore some of the properties of uniformly differentiable functions. Theorem~\ref{thm:UniformDiffRespectsExtension} below is the main tool in the
proof of Theorem~\ref{T1}. In addition, we proudly present a new entry on the list of statements equivalent to completeness, which also involves 
uniform differentiability (Theorem~\ref{Complete}).    

\begin{defn} 
 Let $f$ be a function from $\abF \subset \F $ to $\F$ and let
 \[
  D_c[f](x) = D_x[f](c) \equiv \frac{f(x)-f(c)}{x-c} \qquad (x,c\in \abF,\, x\not=c). 
 \]
 denote the \emph{difference quotient of $f$}.
 \begin{enumerate}[label=(\roman{*}), ref=\roman{*}]
  \item {(Differentiability)}\label{defn:diff'ble}
   \mbox{} \\
   $f$ is \emph{differentiable at $c\in \abF$} if there exists a number $a \in \F$ such that, for every $\epsilon \in \F^+$, there is some $\delta \in \F^+$ 
   such that $|D_c[f](x)-a| < \epsilon$ for every $x \in \abF$ with $0 < |x-c| < \delta$.  If this is the case, we say that $a$ is the \emph{derivative of $f$ at $c$} 
   and write $f'(c) = a$.
  \item {(Uniform Differentiability)}\label{defn:unif-diff'ble}
   \mbox{} \\
    $f$ is \emph{uniformly differentiable on $\abF$} if there exists a function $g:\abF \to \F$ such that for every $\epsilon \in \F^+$, there is a $\delta \in \F^+$ 
    such that $|D_y[f](x)-g(y)| < \epsilon$ for every pair $x,y \in \abF$ with $0 < |x-y| < \delta$. In this case, we say that $g$ is the \emph{derivative of $f$ on $\abF$} 
    and write $f' = g$.
 \end{enumerate} 
\end{defn}
Clearly, $a$ and $g$ are unique and if $f$ is uniformly differentiable, it is differentiable for every $c \in \abF$ with $f'(c) = g(c)$. 
Moreover, the following standard facts may be verified using their standard proofs from Real Analysis.     

\begin{enumerate}[label=(\Roman{*}), ref=(\Roman{*})]
 \item \label{lem:appendA:4}
  $f$ is differentiable at $c$ with value $f'(c)=a$ if and only if there is a function $r_c : \abF \to \F$, continuous at $c$ with $r_c(c) = 0$, such that for all $x \in \abF$,
  $f(x) = f(c) + a(x-c) + r_c(x)(x-c)$.
 \item
  Suppose $f$ is differentiable at $c$. Then $f$ is continuous at $c$.
 \item \label{lem:appendA:3}
  Suppose $f$ is uniformly differentiable. Then $f'$ is uniformly continuous.
\end{enumerate}
The converse of \ref{lem:appendA:3} requires completeness and is in fact equivalent to it:

\begin{thm} \label{Complete}
 $\F $ is complete if and only if every differentiable function whose derivative is uniformly continuous is uniformly differentiable. 
\end{thm}
\begin{proof}
 \RA Since $\F $ is complete, we have $\F =\R $ and the Mean Value Theorem holds. Let $\epsilon \in \F^+$ be given and $\delta \in \F^+$ such that  $|f'(x)-f'(y)| < \epsilon$ 
 for all $x,y\in [a,b] $ such that $|x-y| < \delta$. For arbitrary such $x,y $ (assume w.l.o.g. that $x < y$), choose $w \in [x,y]$ such that $(f(x)-f(y))/(x-y) = f'(w)$. 
 Then $|x-w| < \delta$ and so $\left|D_x[f](y) - f'(x)\right| = |f'(w) - f'(x)| < \epsilon $, which implies that $f$ is uniformly differentiable, as desired.

 \LA Assume that $\F $ is incomplete and let $\co \in \R \setminus \F $.  Then the derivative of the function $f:\F \to \F $ defined by
 \[
  f(x) = \begin{cases}
          0, \mbox{ if } x < \co \\
	  1, \mbox{ if } x > \co
         \end{cases}
 \] 
 is identically zero and hence uniformly continuous. However, by choosing sequences $\{x_n^\pm\} \subset \F$ with $\lim_{n \to \infty } x_n^\pm = \co $ and $x_n^-<\co<x_n^+ $,
 we may show that $f$ cannot be uniformly differentiable, since $|D_{x_n^-}[f](x_n^+) - f(x_n^-) | = 1/(x_n^+-x_n^-) \to \infty $, a contradiction. 
 (Here we used that $\F $ is a dense subset of $\R $, which  is shown in the proof of Lemma~\ref{lem:appendA:16} below.)
\end{proof} 

Before stating the other main result of this section (Theorem~\ref{thm:UniformDiffRespectsExtension}), we record a simple but important fact about uniformly continuous functions.
\begin{lem} \label{lem:appendA:16}
 Let $f$ be uniformly continuous on $\abF$. Then $f$ has a unique (uniformly) continuous extension to $\abR \subset \R $, where $\abR \cap \F = \abF$; 
 i.e. there exists a continuous (hence uniformly continuous) function $\bar{f} : \abR  \to \R $ such that $\bar{f}(x) = f(x)$ for all $x\in \abF$ and $\bar{f}$ 
 is the only function with these properties. In particular, $f$ is bounded\footnote{In light of the assumed (uniform) continuity of $f$, this assertion may seem redundant.
 However, it turns out that one needs to assume completeness of $\F $ to ensure that every continuous function is bounded. Similarly, the property that every \emph{uniformly} 
 continuous function is bounded is equivalent to the field being Archimedean.}.
\end{lem}
\begin{proof}
 It is a standard result of Real Analysis that uniformly continuous functions can (uniquely) be extended to the closure of their domains, so we only need to argue that $\abF$ is a 
 dense subset of $\abR$. However, this immediately follows from the density of $\Q $ in $\R $ and $\Q \subset \F \subset \R $.\footnote{Note that  every ordered field contains a 
 copy of $\Q $.} 
\end{proof}

The next theorem shows that uniform differentiability ``jibes well'' with continuous extension, which is the main reason for requiring it of anti-derivatives in Theorem~\ref{T1}. 

\begin{thm} \label{thm:UniformDiffRespectsExtension}
 Let $f$ be uniformly differentiable on $\abF$. Then $f$ is uniformly continuous and its unique (uniformly) continuous extension $\bar{f} : \abR \to \R$ is uniformly differentiable 
 with $\bar{f}'(x) = f'(x)$ for all $x \in \abF$. Furthermore, the extension of $f'$ exists (uniquely) and coincides with $\bar{f}'$ on $\abR$.
\end{thm}

The proof of this theorem is somewhat lengthy and therefore relegated to the appendix. Instead we proceed to the 

\section{Proof of the main results} \label{Main} 

\begin{proof}[Proof of Theorem~\ref{T1}]
 \RA Let $\F$ be complete, i.e. $\F =\R $, and so the standard Fundamental Theorem of Calculus I holds: if $f: \abF \to \F$ is an arbitrary continuous function, 
 it has a continuous antiderivative, say $F$. Since $f$ is defined on the closed and bounded interval $\abR$, it is uniformly continuous, so the derivative $F' = f$ of $F$
 is uniformly continuous. By Theorem~\ref{Complete}, this implies that $F$ is uniformly differentiable on $\abF$. \\
 \LA We argue by contradiction. Let $\F$ be incomplete and $\co \in \R \setminus \F $.
 Then, by Def.~\ref{ProppFunction} and Lemma~\ref{lem:ProppIntegral}, we have a continuous function $\propp: \oeR \to \R$ such that $\int_0^1 \propp(x)\,dx= \co$. 
 Denote its restriction $\proppF : \oeF \to \F$ to $\oeF$ by $f$. \renewcommand{\proppF}{f} 
 
 By assumption, there is a uniformly differentiable function $\ADF : \oeF \to \F$ such that $\ADF'(x) = \proppF(x)$ for all $x \in \oeR_\F$.
 We may assume w.l.o.g. $\ADF (0) = 0$. Then, by Lemma~\ref{thm:UniformDiffRespectsExtension}, $\ADF$ has a unique extension $\ADR $ to $\oeR$, 
 which is (uniformly) differentiable and whose derivative $\ADR'$ is equal to the continuous extension of $\ADF' = \proppF$, i.e. $\ADR '= \overline{\ADF'} = \bar{\proppF} = \propp$.
 So $\ADR $ is an antiderivative of $\propp$ in $\R$.
 
 Now $\ADF(1) = \ADF(1) - \ADF(0) = \ADR(1) - \ADR(0)$, since $\ADR$ and $\ADF$ must coincide at $0$ and $1$. But $\ADR(1) - \ADR(0) = \int_0^1 \propp(x)\,dx= \co$ 
 by the standard Fundamental Theorem of Calculus II applied to  $\propp$ and $\ADR$, which  are functions in $\R$. So $\ADF(1) = \co$, which is impossible, since $\co \notin \F$.
\end{proof}

\begin{proof}[Proof of Theorem \ref{T2}]
 \RA Standard Real Analysis. \\  \renewcommand{\proppF}{f} 
 \LA Assume again that $\F$ is incomplete. We claim that the restriction $\proppF = \propp_\F $ of the P--function $\propp $ cannot be Riemann-integrable. 
  To see this, assume it is, which means that the limit of right sums, $\lim_{n \to \infty} \frac{1}{n} \sum_{j=1}^n \proppF(\frac{j}{n})$, exists in $\F $; call it $a \in \F$. 
 Since $\proppF(j/n) = \propp(j/n)$ for all $n\in \N $, we obtain 
 \[
  a = \lim_{n\to \infty} \frac{1}{n} \sum_{j=1}^n \proppF\left(\frac{j}{n}\right) = \lim_{n\to \infty} \frac{1}{n} \sum_{j=1}^n \propp\left(\frac{j}{n}\right) 
  = \int_0^1 \propp(x)\,dx= \alpha \in \R \setminus \F,
 \]
 a contradiction since $a \in \F$.
\end{proof}

\section{But wait: what about ``the other'' FTC?} 

So far we have not addressed the second (part) of the FTC\footnote{The numbering of the parts is somewhat inconsistent, but judging from our sample of Calculus text books,
the evaluation part is more often designated as ``Part II''.}, often called the ``Evaluation Theorem'' (ET).  Is it perhaps also equivalent to completeness? 
One hint that this may indeed be the case may be found in the standard proofs of the ET utilizing the Mean Value Theorem (in form or another), which itself is equivalent to completeness.
As the reader is well aware by now, the name of the game of showing the difficult direction (ET $\Rightarrow$ completeness) is to find a counterexample, 
if $\F $ is assumed to be incomplete. Here this amounts to finding an integrable function $f$, possessing an anti-derivative $F$, such that 
\begin{equation} \label{FTC2}
 \int_a^b f(x)\,dx = F(b) - F(a) 
\end{equation}
does \emph{not} hold.       
A moment's reflection reveals that the functions $f \equiv 0 $ and $F$ given by the function $f$ in the proof of Theorem~\ref{Complete} ($\Leftarrow $)
have precisely the required properties. As a result, we obtain
\begin{thm}
 $\F $ is complete if and only if for every Riemann-integrable function $f$ possessing an anti-derivative $F$ the identity \eqref{FTC2} holds.       
\end{thm}
 
One final question: having been sensitized to the utility of the assumption of \emph{uniform} differentiability, we may be tempted to ask what the effect might be of replacing 
``anti-derivative'' with ``\emph{uniformly differentiable} anti-derivative''. The answer is given in the next theorem whose proof is left to the motivated reader. 
(\emph{Hint:} Prove that, in any subfield of the reals, uniformly differentiable functions satisfy an approximate version of the Mean Value Theorem.) 

\begin{thm}
 Let $f$ be a Riemann-integrable function with a uniformly differentiable anti-derivative $F$. Then \eqref{FTC2} holds. 
\end{thm}  

\appendix

\section*{Appendix: Proof of Theorem \ref{thm:UniformDiffRespectsExtension}}

The bulk of the proof broken up into several lemmas, which we list first. 
\begin{lem} \label{lem:appendA:2}
 Suppose $f$ is uniformly differentiable and $f'$ is bounded. Then $f$ is uniformly continuous.
\end{lem}
\begin{proof}
 Let $\epsilon \in \F^+$ be given and $\delta_1 \in \F^+$ be as in \eqref{defn:unif-diff'ble} with $\epsilon / 2$. 
 Moreover, let $\delta = \min\{1,\delta_1,\epsilon/(2M)\}$, where $M$ is a bound on $|f'|$. Then, for  $x,y\in \abF$ 
 such that $|x-y| < \delta$, 
 \begin{align*}
  |f(x) - f(y)| &= |f'(y)(x-y)+\left(D_y[f](x) - f'(y)\right)(x-y)| \\
  &< \delta|f'(y)| + \delta|D_y[f](x) - f'(y)| \\
  &\leq \frac{\epsilon|f'(y)|}{2M} + \frac{\epsilon}{2} = \frac{\epsilon}{2} + \frac{\epsilon}{2} = \epsilon
 \end{align*}
 This shows the uniform continuity of $f$.
\end{proof}

\begin{lem}\label{lem:appendA:5}
 Let $f$ be differentiable at $c$, with $f$ uniformly continuous on $\abF$. Then the function $r_c:\abF \to \F$ defined in \ref{lem:appendA:4} above is uniformly continuous.
\end{lem}
\begin{proof}
 Let $\epsilon \in \F^+$ be given. Then since $f$ is differentiable at $c$, there is  $\delta_1 \in \F^+$ such that
 $|D_c[f](x) - f'(c)| < \epsilon/2$ and so $r_c(x)$ gives $|r_c(x)| < \epsilon/2$  for all $x \in \abF$ with $0 < |x - c| < \delta_1$.
 
 Since $f$ is uniformly continuous on $\abF$, there is  $\delta_2 \in \F^+$ such that $|f(x) - f(y)| < \epsilon\delta_1/4$
 for all $x,y \in \abF$ such that $0<|x-y|<\delta _2$.
 
 Also, $f$ is bounded on $\abF$ by Lemma \ref{Complete}, so let $M \in \F^+$ be a bound for $f$, i.e. $|f(x)| \leq M$ for all $x \in \abF$, and let $\delta_3 = \delta_1^2\epsilon/(16M)$.
 
 Now choose $\delta = \min\{\delta_1/2,\delta_2,\delta_3\}$ and let $x,y \in \abF$ such that $0<|x-y| < \delta$.

 \noindent\textsc{Case I}:  $|x-c| < \delta_1/2$ or $|y-c| < \delta_1/2$.
 
 W.l.o.g., assume the former. Then
 \[
  |y-c| \leq |y-x| + |x-c| < \delta + \frac{\delta_1}{2} \leq \frac{\delta_1}{2} + \frac{\delta_1}{2} = \delta_1
 \]
 
 and so both $|x-c|$ and $|y-c|$ are less than $\delta_1$. Then by definition,
 \[
  |r_c(x)-r_c(y)| \leq |r_c(x)| + |r_c(y)| < \epsilon/2 + \epsilon/2 < \epsilon.
 \]

 \noindent\textsc{Case II}: $|x-c|\ge \delta_1/2$ and $|y-c| \ge \delta_1/2$. 
 
 Then
 \begin{align*}
  |r_c(x)-r_c(y)| &= |D_c[f](x) - D_c[f](y)| = \left|\frac{f(x)-f(c)}{x-c} - \frac{f(y)-f(c)}{y-c} \right| \\
  &= \left|\frac{f(x)-f(y)}{y-c} + (x-y)\frac{f(x)-f(c)}{(x-c)(y-c)}\right| \\
  &\leq \frac{|f(x)-f(y)|}{|y-c|} + |x-y|\frac{|f(x)|+|f(c)|}{|x-c||y-c|} \\
  &< \frac{\delta_1\epsilon/4}{\delta_1/2} + \frac{\delta_1^2\epsilon}{16M}\frac{2M}{(\delta_1/2)^2} = \frac{\epsilon}{2} + \frac{\epsilon}{2} = \epsilon .
 \end{align*}
 So in both cases, $|r_c(x)-r_c(y)| < \epsilon$ for all $|x-y| < \delta$, which shows that $r_c$ is uniformly continuous on $\abF$.
\end{proof}

\begin{lem} \label{lem:appendA:6}
 Let $f$ be differentiable at $c\in \abF$ and uniformly continuous on $\abF$. Then the unique uniformly continuous extension 
 $\bar{f}:\abR \to \R $ of $f$ is differentiable at $c$ with $f'(c) = \bar{f}'(c)$.
\end{lem}
\begin{proof}
 Since $f$ is uniformly continuous, it is bounded by Lemma~\ref{lem:appendA:16} and hence $r_c(x):\abF \to \F$ is uniformly continuous by Lemma~\ref{lem:appendA:5}.
 Therefore, $r_c$ has a unique uniformly continuous extension, $\bar{r}_c : \abR \to \R$. Clearly, $\bar{r}_c(0) =r_c(0) = 0$.
 
 As a result, the function $g : \abR \to \R$, defined by $g(x) = f(c) + f'(c)(x-c) + \bar{r}_c(x)(x-c)$, is uniformly continuous as well. 
 Furthermore, $g(x) = f(c) + f'(c)(x-c) + \bar{r}_c(x)(x-c) = f(c) + f'(c)(x-c) + r_c(x)(x-c) = f(x)$  for all $x \in \abF$, 
 so $g$ agrees with $f$ on $\abF$. By the uniqueness of the continuous extension, we get $\bar{f}(x) = g(x) = f(c) + f'(c)(x-c) + \bar{r}_c(x)(x-c)$, 
 which, in light of \ref{lem:appendA:4}, means that $\bar{f}$ is differentiable at $c$ with $\bar{f}'(c) = f'(c)$.
\end{proof}

We are now finally ready for the

\begin{proof}[Conclusion of the proof of Theorem \ref{thm:UniformDiffRespectsExtension}]
 First, by \ref{lem:appendA:3}, $f'$ is uniformly continuous and therefore has a unique uniformly continuous extension $\overline{f'}$. 
 Moreover, $f'$ is bounded by Lemma~\ref{lem:appendA:16} and hence $f$ is uniformly continuous by Lemma~\ref{lem:appendA:2}, i.e. $\bar{f}:\abR \to \R$ exists.
 
 We will show that $\bar{f}$ is uniformly differentiable on $\abR$ and $\bar{f}' = \overline{f'}$. To this end, let $\epsilon \in \R^+$ be given. Since $\F$ is dense in $\R$,
 we can assume w.l.o.g. $\epsilon \in \F^+$.
 
 Since $f$ is uniformly  differentiable, we can find a $\delta \in \F^+$ such that for all $x,y \in \abF$ with $0 < |x-y| < \delta$, $|D_y[f](x) - f'(y)| < \epsilon/2$. 
 
 Now let $p,q \in \abR$ be such that $0 < |p-q| < \delta$; w.l.o.g. assume that $p < q$. Since $\Q$ is dense in $\R$, there are sequences $\{p_n\}$ and $\{q_n\}$ 
 drawn from $\F$ convergent to $p$ and $q$, respectively. We may assume w.l.o.g. that $p_n,q_n \in \abF$ and $0 < |p_n-q_n| < \delta$ for all $n \in \N$, 
 which implies $|D_{q_n}[f](p_n) - f'(p_n)| < \epsilon/2$. 
 
 Then by definition and the continuity of $\bar{f}$ and $\overline{f'}$, we have:
 \begin{align*}
  |D_q[\bar{f}](p) - \overline{f'}(p)| &= \left|\lim_{n \to \infty}\left(\frac{\bar{f}(p_n)-\bar{f}(q_n)}{p_n-q_n} - \overline{f'}(p_n)\right)\right|
    \intertext{Moreover, $\bar{f}(p_n) = f(p_n)$, $\bar{f}(q_n) = f(q_n)$,  and $\overline{f'}(p_n) = f'(p_n)$, since $p_n,q_n \in \abF$, and so}
   &= \lim_{n \to \infty}\left|D_{q_n}[f](p_n) - f'(p_n)\right| \leq \frac{\epsilon}{2} < \epsilon .
 \end{align*}
  Since our choice of $\delta$ was independent of $p$ and $q$, this shows that $\bar{f}$ is uniformly differentiable, with $\bar{f}' \equiv \overline{f'}$, as desired.
\end{proof}

\bibliographystyle{abbrv}

\end{document}